\documentclass[11pt,reqno]{amsart}
\usepackage[T1]{fontenc}
\usepackage{lmodern}
\usepackage{amsmath,amssymb}
\usepackage{microtype}
\usepackage[letterpaper,margin=1in]{geometry}
\usepackage[hidelinks]{hyperref}

\newtheorem{theorem}{Theorem}[section]
\newtheorem{lemma}[theorem]{Lemma}
\newtheorem{corollary}[theorem]{Corollary}
\numberwithin{equation}{section}
\allowdisplaybreaks[1]

\title[A parametric Catalan-type congruence]{A parametric Catalan-type congruence for generalized central trinomial coefficients}
\hypersetup{pdftitle={A parametric Catalan-type congruence for generalized central trinomial coefficients},pdfauthor={}}

\author{Yuqi Liu}
\address{School of Mathematical Sciences, Capital Normal University, Beijing 100048, P.R.~China}
\email{lyq204219@163.com}
\date{}

\begin{document}
\begin{abstract}
Let $p>3$ be a prime, and let $b,c$ be integers with $p\nmid b+2c$. We obtain an explicit congruence modulo $p^2$ for the sum of $\binom{2k}{k}T_{2k}(b,c^2)/(4^k(k+1)(b+2c)^{2k})$ from $k=0$ to $(p-1)/2$, where $T_n(b,c)$ denotes the generalized central trinomial coefficient. This gives a parametric formula for the sum considered in a question of Wang and Cui. The proof uses a constant term representation and a logarithmic generating function to reduce the sum to two coefficients of an algebraic generating function. A differential identity and a coefficientwise polynomial congruence then yield the desired evaluation. The formula also applies when $p$ divides $b$ or $b-2c$.
\end{abstract}
\maketitle

\section{Introduction and notation}

For a nonnegative integer $n$, the generalized central trinomial coefficient (see, e.g., \cite{Noe2006}) is defined by

\[
T_n(b,c)=[x^n](x^2+bx+c)^n
=\operatorname{CT}_x(b+x+cx^{-1})^n.\tag{1.1}\label{eq:1.1}
\]

Here $[x^n]$ denotes coefficient extraction, and $\operatorname{CT}_x$ denotes the constant term with respect to $x$. We study the truncated sum

\[
S_p(b,c):=\sum_{k=0}^{(p-1)/2}
\frac{\binom{2k}{k}\,T_{2k}(b,c^2)}
{4^k(k+1)(b+2c)^{2k}}.\tag{1.2}\label{eq:1.2}
\]

The factor $\binom{2k}{k}/(k+1)$ is the $k$th Catalan number. The specialization $(b,c)=(4,-1)$ of \eqref{eq:1.2} is the sum appearing in a congruence conjectured by Sun \cite[Conjecture~2.1]{Sun2014} and proved by Wang and Cui \cite[Theorem~1.2]{WangCui2026}. Immediately after that theorem on page~3 of the cited version, Wang and Cui asked for a determination of \eqref{eq:1.2} modulo $p^2$.

In this paper, we give an explicit formula for \eqref{eq:1.2} in terms of the single parameter $a=(b-2c)/(b+2c)$. The proof combines the constant term method with formal generating functions. Its main step is a polynomial congruence for an auxiliary coefficient polynomial $J_p(a)$, from which the truncated sum is obtained by differentiation.

Throughout, $p>3$ is a prime, and we put

\[
m=\frac{p-1}{2},\qquad \varepsilon=(-1)^m.
\]

We write $\mathbb Z_{(p)}=\{u/v:\ u,v\in\mathbb Z,\ p\nmid v\}$. Congruences between rational numbers are interpreted in $\mathbb Z_{(p)}$, and polynomial congruences are interpreted coefficientwise.
All generating functions are expanded as formal power series in $s$
or $t$, as appropriate. Their coefficients are polynomials in $z$
whenever $z$ occurs, and Laurent polynomials in $x$ whenever $x$ occurs. Every square root is chosen to have constant term $1$. Expressions involving both $x$ and $t$ are first expanded in $t$, after which $\operatorname{CT}_x$ is applied coefficientwise.

Section 2 establishes four auxiliary lemmas. The main theorem and its proof are given in Section 3, followed by two specializations in Section 4.

\section{Auxiliary lemmas}

\begin{lemma}\label{lemma:2.1}
For $n\geq0$, define

\[
R_n(z)=\sum_{k=0}^{n}\frac{(-1)^k}{k+1}
\binom{n+k}{2k}z^{2k}.\tag{2.1}\label{eq:2.1}
\]

Then

\[
\sum_{n=0}^{\infty}R_n(z)s^n
=\frac{1-s}{z^2s}
\log\left(1+\frac{z^2s}{(1-s)^2}\right).\tag{2.2}\label{eq:2.2}
\]

\end{lemma}

\begin{proof}
For each fixed power of $s$, only finitely many summands contribute. Thus we may interchange the summations to obtain

\[
\sum_{n\geq0}R_n(z)s^n
=\sum_{k\geq0}\frac{(-1)^kz^{2k}}{k+1}
  \sum_{n\geq k}\binom{n+k}{2k}s^n.
\]

Writing $n=k+r$ and using the binomial series gives

\[
\sum_{n\geq k}\binom{n+k}{2k}s^n
=s^k\sum_{r\geq0}\binom{r+2k}{2k}s^r
=\frac{s^k}{(1-s)^{2k+1}}.
\]

Consequently, with $u=z^2s/(1-s)^2$, we have

\[
\sum_{n\geq0}R_n(z)s^n
=\frac{1}{1-s}\sum_{k\geq0}\frac{(-1)^k u^k}{k+1}
=\frac{1}{1-s}\frac{\log(1+u)}{u}.
\]

The last equality follows from the formal expansion of $\log(1+u)$. Substituting the value of $u$ proves \eqref{eq:2.2}. At $z=0$, its formal expansion is $1/(1-s)$, in agreement with $R_n(0)=1$.
\end{proof}

For the next lemmas, regard $a$ as an indeterminate and set

\[
\begin{gathered}
V_a(x)=(a+1)+\frac{1-a}{2}(x+x^{-1}),\\
F_a(t)=(1-2at+t^2)^{-1/2},\\
D_a(t)=\frac{t^2(1+t)}{(1+t^2)^2}F_a(t)
       =\sum_{r\geq0}d_r(a)t^r.
\end{gathered}\tag{2.3}\label{eq:2.3}
\]

In particular, $d_0(a)=d_1(a)=0$.

\begin{lemma}\label{lemma:2.2}
With the notation in \eqref{eq:2.1} and \eqref{eq:2.3}, we have

\[
\begin{gathered}
\sum_{n\geq0}(-1)^n\operatorname{CT}_xR_n(V_a(x))t^{2n}
=\frac{1+t^2}{t^2}\sum_{r\geq1}\frac{d_{2r}(a)}{r}t^{2r}\\
=\sum_{r\geq1}\frac{d_{2r}(a)}{r}t^{2r-2}
 +\sum_{r\geq1}\frac{d_{2r}(a)}{r}t^{2r}.
\end{gathered}\tag{2.4}\label{eq:2.4}
\]

In particular, for every $n\geq1$,

\[
\operatorname{CT}_xR_n(V_a(x))
=(-1)^n\left(\frac{d_{2n+2}(a)}{n+1}
             +\frac{d_{2n}(a)}{n}\right).\tag{2.5}\label{eq:2.5}
\]
\end{lemma}

\begin{proof}
We first record a constant term identity. Let $\alpha$ and $\beta$ be formal power series in $t$ with $\alpha(0)=1$ and $\beta(0)=0$. Expanding in $t$, we obtain

\[
\begin{gathered}
\operatorname{CT}_x\frac{1}{\alpha-\beta(x+x^{-1})}
=\frac1\alpha\sum_{r\geq0}\binom{2r}{r}
  \left(\frac\beta\alpha\right)^{2r}\\
=\frac{1}{\sqrt{\alpha^2-4\beta^2}}.
\end{gathered}\tag{2.6}\label{eq:2.6}
\]

Indeed, $\operatorname{CT}_x(x+x^{-1})^{2r}=\binom{2r}{r}$, whereas the constant term of an odd power is zero. Take

\[
\alpha=1-(a+1)t+t^2,\qquad \beta=\frac{1-a}{2}t.
\]

Then

\[
\alpha^2-4\beta^2=(1-t)^2(1-2at+t^2).
\]

Applying \eqref{eq:2.6}, and then replacing $t$ by $-t$, gives

\[
\operatorname{CT}_x\frac{1}{1-V_a t+t^2}
=\frac{F_a(t)}{1-t},\qquad
\operatorname{CT}_x\frac{1}{1+V_a t+t^2}
=\frac{F_a(-t)}{1+t}.\tag{2.7}\label{eq:2.7}
\]

Write $V=V_a(x)$ and define

\[
L(t,x)=\sum_{r\geq1}\frac{V^{2r-2}t^{2r}}{r(1+t^2)^{2r}}
=-\frac1{V^2}\log\left(1-\frac{V^2t^2}{(1+t^2)^2}\right).
\tag{2.8}\label{eq:2.8}
\]

The series on the left defines $L$ without requiring $V$ to be invertible. Substituting $s=-t^2$ and $z=V$ into Lemma 2.1 yields

\[
\sum_{n\geq0}(-1)^nR_n(V)t^{2n}
=\frac{1+t^2}{t^2}L(t,x).
\]

Differentiating \eqref{eq:2.8} with respect to $t$ and decomposing the denominator, we find

\[
\frac t2\frac{\partial L}{\partial t}
=\frac{t^2(1-t^2)}{2(1+t^2)^2}
\left(\frac1{1-Vt+t^2}+\frac1{1+Vt+t^2}\right).
\]

Taking constant terms in $x$ and using \eqref{eq:2.7}, we obtain

\[
\begin{gathered}
\operatorname{CT}_x\frac t2\frac{\partial L}{\partial t}
=\frac{t^2}{2(1+t^2)^2}
  \bigl((1+t)F_a(t)+(1-t)F_a(-t)\bigr)\\
=\frac{D_a(t)+D_a(-t)}2
=\sum_{r\geq1}d_{2r}(a)t^{2r}.
\end{gathered}\tag{2.9}\label{eq:2.9}
\]

Since $L(0,x)=0$ and $\frac t2\frac{d}{dt}t^{2r}=r t^{2r}$, coefficient comparison gives

\[
\operatorname{CT}_xL(t,x)=\sum_{r\geq1}\frac{d_{2r}(a)}r t^{2r}.
\]

Substitution proves \eqref{eq:2.4}. Comparing the coefficients of $t^{2n}$ for $n\geq1$ gives \eqref{eq:2.5}. For $n=0$, only $d_2(a)=1$ contributes, so no expression containing $1/n$ is needed.
\end{proof}

\begin{lemma}\label{lemma:2.3}
Define

\[
J_p(a)=[t^{p-1}]\frac{1-t}{1+t^2}F_a(t).\tag{2.10}\label{eq:2.10}
\]

Then the following identity holds in $\mathbb Q[a]$:

\[
(p-1)d_{p+1}(a)+(p+1)d_{p-1}(a)
=(a+1)J_p'(a),\tag{2.11}\label{eq:2.11}
\]

where the prime denotes differentiation with respect to $a$.
\end{lemma}

\begin{proof}
Direct differentiation gives

\[
\frac{\partial F_a}{\partial a}=tF_a^3,
\qquad
\frac{\partial F_a}{\partial t}=(a-t)F_a^3.
\]

It follows that

\[
\frac{d}{dt}\bigl((1+t)F_a(t)\bigr)
=F_a+(1+t)(a-t)F_a^3
=(a+1)(1-t)F_a^3.
\]

Applying the product rule to $\frac{t}{1+t^2}\bigl((1+t)F_a(t)\bigr)$, we obtain

\[
\begin{gathered}
\frac{d}{dt}\left(\frac{t(1+t)}{1+t^2}F_a(t)\right)
=\frac{(1-t^2)(1+t)}{(1+t^2)^2}F_a(t)\\
\quad +(a+1)\frac{\partial}{\partial a}
  \left(\frac{1-t}{1+t^2}F_a(t)\right).
\end{gathered}\tag{2.12}\label{eq:2.12}
\]

The expression inside the derivative on the left is $(1+t^2)D_a(t)/t$. Its coefficient of $t^p$ is $d_{p+1}(a)+d_{p-1}(a)$, so the coefficient of $t^{p-1}$ after differentiation is

\[
p\bigl(d_{p+1}(a)+d_{p-1}(a)\bigr).
\]

The first term on the right of \eqref{eq:2.12} is $(1-t^2)D_a(t)/t^2$, whose coefficient of $t^{p-1}$ is $d_{p+1}(a)-d_{p-1}(a)$. The corresponding coefficient of the second term is $(a+1)J_p'(a)$. Hence

\[
p\bigl(d_{p+1}(a)+d_{p-1}(a)\bigr)
=d_{p+1}(a)-d_{p-1}(a)+(a+1)J_p'(a).
\]

Rearranging proves \eqref{eq:2.11}.
\end{proof}

\begin{lemma}\label{lemma:2.4}
Let $p>3$ be a prime and $m=(p-1)/2$. The polynomial in \eqref{eq:2.10} satisfies

\[
J_p(a)\equiv(-1)^m a^m
+p\!\sum_{\substack{0\leq j\leq p-1\\j\ne m}}
\frac{(-a)^j}{2j+1}\pmod{p^2}.\tag{2.13}\label{eq:2.13}
\]

Moreover,

\[
J_p'(a)\equiv(-1)^m m a^{m-1}
-p\!\sum_{\substack{1\leq j\leq p-1\\j\ne m}}
\frac{j(-a)^{j-1}}{2j+1}\pmod{p^2}.\tag{2.14}\label{eq:2.14}
\]

Both congruences hold coefficientwise in $\mathbb Z_{(p)}[a]$.
\end{lemma}

\begin{proof}
The binomial series gives

\[
F_a(t)=\sum_{n\geq0}\frac{\binom{2n}{n}}{4^n}(2at-t^2)^n,
\]

so every coefficient of $F_a(t)$ belongs to $\mathbb Z[1/2,a]$. Since $(2at-t^2)^n=t^n(2a-t)^n$, only $n\leq p-1$ can contribute to the coefficient of $t^{p-1}$. Each such term has degree at most $n$ in $a$. Multiplication by $(1-t)/(1+t^2)$, which is independent of $a$ and has no negative powers of $t$, preserves this degree bound. Thus

\[
J_p(a)=\sum_{j=0}^{p-1}C_j a^j,
\qquad C_j\in\mathbb Z_{(p)}.
\]

Using $F_a+2a\,\partial F_a/\partial a=(1+t^2)F_a^3$ and the differentiation identity in the proof of Lemma 2.3, we obtain

\[
\begin{gathered}
2aJ_p'(a)+J_p(a)=[t^{p-1}](1-t)F_a(t)^3,\\
(a+1)\bigl(2aJ_p'(a)+J_p(a)\bigr)
=p[t^p](1+t)F_a(t).
\end{gathered}
\]

To evaluate the two required coefficients of $F_a$, put $Q(t)=(1-2at+t^2)^m$. The equality $p=2m+1$ gives the exact identity

\[
F_a(t)=Q(t)F_a(t)^p.
\]

As $F_a(t)=1+at+\cdots$, the Frobenius identity modulo $p$ yields

\[
F_a(t)^p\equiv1+a^pt^p\pmod{(p,t^{p+1})}.
\]

Here the congruence is in the formal power series ring modulo the ideal generated by $p$ and $t^{p+1}$. The polynomial $Q(t)$ has degree $p-1$ in $t$, and both its constant coefficient and its leading coefficient are $1$. Therefore

\[
[t^{p-1}]F_a(t)\equiv1,\qquad
[t^p]F_a(t)\equiv a^p\pmod p.
\]

It follows that $[t^p](1+t)F_a(t)\equiv1+a^p\equiv(a+1)^p\pmod p$. Substitution into the exact identity above gives

\[
(a+1)\bigl(2aJ_p'(a)+J_p(a)\bigr)
\equiv p(a+1)^p\pmod{p^2}.
\]

This is a polynomial congruence in the indeterminate $a$. Since the monic polynomial $a+1$ is not a zero divisor in $(\mathbb Z_{(p)}/p^2\mathbb Z_{(p)})[a]$, it can be cancelled. We conclude that

\[
2aJ_p'(a)+J_p(a)\equiv p(a+1)^{p-1}\pmod{p^2}.
\tag{2.15}\label{eq:2.15}
\]

This cancellation takes place in a polynomial ring and imposes no invertibility condition on later specializations of $a+1$. For $0\leq j\leq p-1$, we have

\[
\binom{p-1}{j}=\prod_{r=1}^{j}\frac{p-r}{r}\equiv(-1)^j\pmod p.
\]

Comparing coefficients of $a^j$ in \eqref{eq:2.15} therefore yields

\[
(2j+1)C_j\equiv p\binom{p-1}{j}
\equiv p(-1)^j\pmod{p^2}.\tag{2.16}\label{eq:2.16}
\]

In the range $0\leq j\leq p-1$, the integer $2j+1$ is divisible by $p$ precisely when $j=m$. Hence, for $j\ne m$,

\[
C_j\equiv\frac{p(-1)^j}{2j+1}\pmod{p^2}.
\]

To determine the remaining coefficient $C_m$, evaluate at $a=-1$. Since $F_{-1}(t)=1/(1+t)$,

\[
\begin{gathered}
J_p(-1)
=[t^{p-1}]\frac{1-t}{(1+t)(1+t^2)}\\
=[t^{p-1}]\left(\frac1{1+t}-\frac{t}{1+t^2}\right)=1,
\end{gathered}
\]

where the last equality uses the fact that $p-1$ is even. Substituting the coefficients already determined gives

\[
1\equiv(-1)^m C_m+
p\!\sum_{\substack{0\leq j\leq p-1\\j\ne m}}\frac1{2j+1}
\pmod{p^2}.
\]

Pair the index $j$ with $p-1-j$. The only possible fixed point is $j=m$, which has been excluded, and each pair satisfies

\[
\frac1{2j+1}+\frac1{2p-(2j+1)}\equiv0\pmod p.
\]

Thus the reciprocal sum vanishes modulo $p$, and $C_m\equiv(-1)^m\pmod{p^2}$. Together with the other coefficients, this proves \eqref{eq:2.13}. Differentiating coefficientwise and using $\frac{d}{da}(-a)^j=-j(-a)^{j-1}$ proves \eqref{eq:2.14}. The differentiated sum starts at $j=1$, so it remains a polynomial expression at $a=0$.
\end{proof}

\section{The main congruence}

\begin{theorem}\label{theorem:3.1}
Let $p>3$ be a prime, and let $b,c\in\mathbb Z$ satisfy $p\nmid b+2c$. Put

\[
m=\frac{p-1}{2},\qquad \varepsilon=(-1)^m,
\qquad a=\frac{b-2c}{b+2c}.
\]

Then the sum in \eqref{eq:1.2} satisfies

\[
\begin{gathered}
S_p(b,c)\equiv{}-2m(a+1)a^{m-1}\\
+2\varepsilon p(a+1)
\!\sum_{\substack{1\leq j\leq p-1\\j\ne m}}
\frac{j(-a)^{j-1}}{2j+1}\pmod{p^2}.
\end{gathered}\tag{3.1}\label{eq:3.1}
\]

Equivalently, the first term on the right is $(1-p)(a+1)a^{m-1}$. Every denominator $2j+1$ appearing in the sum is coprime to $p$, so the expression is defined under the stated hypothesis.
\end{theorem}

\begin{proof}
We begin with the symmetric constant term representation

\[
T_{2k}(b,c^2)=\operatorname{CT}_x\bigl(b+c(x+x^{-1})\bigr)^{2k}.
\]

Both sides expand to

\[
\sum_{r=0}^{k}\binom{2k}{2r}\binom{2r}{r}b^{2k-2r}c^{2r},
\]

so the identity also holds when $c=0$; see also \cite[Lemma~2.1]{WangCui2026}. The definition of $a$ gives

\[
V_a(x)=\frac{2\bigl(b+c(x+x^{-1})\bigr)}{b+2c}.
\]

For $0\leq k\leq m$, pairing factors in the binomial coefficient yields

\[
\begin{gathered}
(-1)^k\binom{m+k}{2k}
=\frac{\prod_{r=1}^{k}\bigl((2r-1)^2-p^2\bigr)}{4^k(2k)!}\\
\equiv\frac{\prod_{r=1}^{k}(2r-1)^2}{4^k(2k)!}
=\frac{\binom{2k}{k}}{16^k}\pmod{p^2}.
\end{gathered}\tag{3.2}\label{eq:3.2}
\]

See also \cite[Eq.~(1.1)]{SunLegendreII2013} for this binomial congruence.

The denominators are coprime to $p$ because $2k\leq p-1$; at $k=0$, the products are empty and equal to $1$. Moreover, $k+1<p$, so division by $k+1$ preserves the congruence. By linearity of the constant term operator,

\[
\begin{gathered}
S_p(b,c)
=\operatorname{CT}_x\sum_{k=0}^{m}
  \frac{\binom{2k}{k}}{16^k(k+1)}V_a(x)^{2k}\\
\equiv\operatorname{CT}_x R_m(V_a(x))\pmod{p^2}.
\end{gathered}\tag{3.3}\label{eq:3.3}
\]

Applying Lemma 2.2 with $n=m$ and using $2m=p-1$, we obtain the exact identity

\[
\operatorname{CT}_xR_m(V_a(x))
=\varepsilon\left(\frac{d_{p+1}(a)}{m+1}
                  +\frac{d_{p-1}(a)}m\right).\tag{3.4}\label{eq:3.4}
\]

Each $d_r(a)$ belongs to $\mathbb Z[1/2,a]$, and its value at $a=(b-2c)/(b+2c)$ belongs to $\mathbb Z_{(p)}$. We may therefore use

\[
\frac1{m+1}=\frac2{p+1}\equiv2(1-p),\qquad
\frac1m=\frac2{p-1}\equiv-2(1+p)\pmod{p^2}.
\]

Suppressing the argument $a$ in $d_r(a)$ for brevity, \eqref{eq:3.3} and \eqref{eq:3.4} give

\[
\begin{gathered}
S_p(b,c)
\equiv2\varepsilon\bigl((d_{p+1}-d_{p-1})
                  -p(d_{p+1}+d_{p-1})\bigr)\\
=-2\varepsilon(a+1)J_p'(a)\pmod{p^2},
\end{gathered}\tag{3.5}\label{eq:3.5}
\]

where the last equality is a rearrangement of Lemma 2.3. Finally, \eqref{eq:2.14} gives

\[
\begin{gathered}
-2\varepsilon(a+1)J_p'(a)
\equiv{}-2m(a+1)a^{m-1}\\
+2\varepsilon p(a+1)
\!\sum_{\substack{1\leq j\leq p-1\\j\ne m}}
\frac{j(-a)^{j-1}}{2j+1}\pmod{p^2}.
\end{gathered}
\]

This proves \eqref{eq:3.1}. The argument requires only $p\nmid b+2c$ and involves no division by $a$ or $a+1$. In particular, it covers the cases $p\mid b-2c$ and $p\mid b$.
\end{proof}

\section{Two specializations}

For completeness, we record two immediate specializations. They illustrate that Theorem 3.1 remains applicable at $a=-1$ and $a=0$.

\begin{corollary}\label{corollary:4.1}
For every prime $p>3$,

\[
\sum_{k=0}^{(p-1)/2}
\frac{\binom{2k}{k}^{\!2}}{16^k(k+1)}
\equiv0\pmod{p^2}.\tag{4.1}\label{eq:4.1}
\]
\end{corollary}

\begin{proof}
Take $(b,c)=(0,1)$ in Theorem 3.1. Then $a=-1$ and $T_{2k}(0,1)=\binom{2k}{k}$. The right-hand side of \eqref{eq:3.1} vanishes because it contains the factor $a+1$, proving \eqref{eq:4.1}.
\end{proof}

\begin{corollary}\label{corollary:4.2}
For every prime $p>3$,

\[
\sum_{k=0}^{(p-1)/2}
\frac{\binom{2k}{k}\binom{4k}{2k}}{64^k(k+1)}
\equiv\frac{2p}{3}(-1)^{(p-1)/2}\pmod{p^2}.
\tag{4.2}\label{eq:4.2}
\]
\end{corollary}

\begin{proof}
Take $(b,c)=(2,1)$. Then $a=0$, and

\[
T_{2k}(2,1)=[x^{2k}](x+1)^{4k}=\binom{4k}{2k}.
\]

Since $m\geq2$, the first term on the right of \eqref{eq:3.1} vanishes. Only the summand $j=1$ survives in the second term, and $j=1\ne m$. The resulting expression is $2\varepsilon p/3$, which proves \eqref{eq:4.2}.
\end{proof}

\end{document}